\documentclass[a4paper,reqno]{amsart}
\usepackage[T1]{fontenc}
\usepackage{lmodern}
\usepackage{mathrsfs}
\usepackage{amsmath,amssymb,amsthm}
\usepackage{mathtools}

\usepackage{upref}
\usepackage{enumerate}
\usepackage{indentfirst}

\usepackage{hyperref}%
\hypersetup{
hypertexnames=false, 
linkbordercolor=[rgb]{0.5, 1, 1},%
urlbordercolor=[rgb]{0.5, 1, 1},%
citebordercolor=[rgb]{0.5, 1, 0.5}}%
\usepackage{graphicx}
\usepackage{subcaption}

\usepackage{cite}
\theoremstyle{plain}
\newtheorem{theorem}{Theorem}
\newtheorem{lemma}[theorem]{Lemma}
\newtheorem{corollary}{Corollary}

\theoremstyle{remark}
\newtheorem{remark}{Remark}

\theoremstyle{definition}
\newtheorem{definition}{Definition}

\begin{document}

\title[Supercritical Branching Random Walks on Nested
Lattices]{Supercritical Branching Random Walks on Nested Lattices}

\author{Elena Yarovaya}


\address{Department of Probability Theory, Lomonosov Moscow State
University, Moscow, 119234, Russia}

\email{elena.yarovaya@math.msu.ru}



\date{}

\begin{abstract}
  In classical models of supercritical continuous-time branching random 
  walks (BRWs) on a $d$-dimensional lattice, the dimension $d$ is typically 
  assumed to be fixed. When comparing the exponential growth rate 
  $\lambda_0$ of a particle population in a BRW with a single branching 
  source across lattices of different dimensions, a question arises whether 
  it is possible to construct a family of random walks on $d$-dimensional 
  lattices such that the model for a $(d+1)$-dimensional lattice 
  ``inherits'' the properties of the model for a $d$-dimensional lattice in 
  a meaningful sense. To address this question, the paper introduces the 
  concept of a ``random walk on nested lattices'' and investigates the 
  effect of the dimension $d$ on particle population growth using the 
  ``inheritance'' condition. For an arbitrary symmetric random walk, a 
  limit theorem is proved on the representation of the positive eigenvalue 
  of the evolution operator for the mean particle numbers in terms of the 
  process parameters and the dimension $d$, as the branching source 
  intensity $\beta \to \infty$. The theorem shows that the influence  of 
  dimension persists even for very strong branching.
\end{abstract} 

\maketitle
\section{Introduction}
Traditionally, models of branching random walks (BRW) on the lattice
$\mathbb{Z}^{d}$ are considered for a fixed lattice dimension $d$. At the
same time, ``dimension-inde\-pendent'' random walk models defined on lattices
of arbitrary dimension are well known---for instance, the random walk model
defined by the Laplace operator. For this model, the dependence of the random
walk operators' eigenvalues on the lattice dimension was investigated
in~\cite{YarGus}. The research method proposed in~\cite{YarGus} relied
heavily on the specific structure of the Laplace operator.

This raises the question: is it possible to describe---in ``reasonable
terms'' and in a ``sufficiently constructive and simple'' manner---a family
of random walk models on the lattices $\mathbb{Z}^{d}$ ($d\in\mathbb{N}$)
such that the model on $\mathbb{Z}^{d+1}$ inherits, in some natural sense,
the properties of the model on the lower-dimensional lattice
$\mathbb{Z}^{d}$, and such that all these models can be viewed as
realizations of a single ``dimension-independent'' model?

The article is structured as follows. Section~\ref{S:RW} introduces key
definitions regarding families of random walks consistent across nested
lattices, enabling the study of the ``inheritance'' of process properties
when transitioning to a lattice of higher dimensionality. Section~\ref{S:0}
reviews the properties of the Green's function, which represents the expected
time a random walk spends at a point in $\mathbb{Z}^{d}$. Section~\ref{S:5}
outlines the description of a supercritical branching random walk (BRW) with
a single particle generation source located at a point in $\mathbb{Z}^{d}$
(an inhomogeneous branching medium). Section~\ref{S:4} presents proofs of the
main results: Theorem~\ref{T:6}, concerning the monotonic decrease of the
supercritical BRW evolution operator's eigenvalues with respect to the
dimension $d$, and Theorem~\ref{T:7}, regarding the asymptotic expansion of
the principal eigenvalue for a supercritical BRW with a high-intensity
particle reproduction source. The concluding section,~\ref{Concl}, summarizes
the main results and discusses the impact of the spatial dimension $d$ on
particle population growth.

\section{On the Hereditary Properties of Random Walks}\label{S:RW}
Let us recall the necessary definitions and properties of a continuous-time
symmetric random walk on $\mathbb{Z}^{d}$ ($d\in \mathbb{N}$) without
particle generation at lattice sites (i.e., without branching); see, for
example,~\cite{Yar2025}. We are interested in random walks on lattices
$\mathbb{Z}^d$ of various dimensions $d$. Therefore, we reformulate the
axiomatic framework for random walks from~\cite{Yar2025} to account for this
factor: a random walk on the lattice $\mathbb{Z}^d$ is defined by a
transition rate matrix $A_{d}=(a_{d}(x,y))_{x,y\in \mathbb{Z}^d}$, the
elements of which satisfy the following conditions:
\begin{itemize}
  \item[$1^{(d)}$.] ${a_{d}(x,y)\ge0}$ for $x\ne y$ and $a_{d}(x,x)<0$,
  \item[$2^{(d)}$.] $\sum_{y} a_{d}(x,y)=0$,
  \item[$3^{(d)}$.] $a_{d}(x,y)=a_{d}(y,x)$, i.e., the random walk is
    symmetric,
  \item[$4^{(d)}$.] the random walk is irreducible, i.e., all points
    $y\in\mathbb{Z}^{d}$ are reachable,
  \item[$5^{(d)}$.] $\sigma^{2}=- \frac{1}{a_{d}(0,0)}\sum_{x\in
    \mathbb{Z}^{d}} |x|^{2}a_{d}(0,x)< \infty$, i.e., the jumps have finite
    variance,
  \item[$6^{(d)}$.] $a_{d}(x,y)=a_{d}(0,y-x)$, i.e., the random walk is
    spatially homogeneous.
\end{itemize}
\begin{remark}\label{R:1}
  It will be convenient to denote the function $a_{d}(0,x)$ by omitting the
  first zero argument:
  \[
    a_{d}(x):=a_{d}(0,x),\qquad x\in\mathbb{Z}^{d},
  \]
  and in certain cases, it will be convenient to use coordinate notation for
  this function:
  \[
    a_{d}(x)\equiv a_{d}(\underbrace{x_{1},x_{2},\ldots,x_{d}}_{d~
    \text{arguments}}),\qquad
    x=(x_{1},x_{2},\ldots,x_{d})\in\mathbb{Z}^{d}.
  \]
\end{remark}
\begin{definition}\label{def:multiRW}
  We say that random walks with intensity matrices
  \[
    A_{d}=(a_{d}(x,y))_{x,y\in
    \mathbb{Z}^d}\quad, d=1,2,\ldots\,,
  \]
  satisfying axioms $1^{(d)}$--$6^{(d)}$, form a \emph{nested
  lattice-consistent set of random walks} (or, more briefly,
    \emph{random walks
  on nested lattices}) if, for each $d=1,2,\ldots$, the relation
  \begin{equation}\label{eq:multidim}
    a_{d+1}(x_{1},\ldots,x_{d},0)= a_{d}(x_{1},\ldots,x_{d}),\qquad
    \forall~x=(x_{1},\ldots,x_{d})\neq0,
  \end{equation}
  holds.

  In the case where, for each $d=1,2,\ldots$, the more general relation
  \begin{equation}\label{eq:multidim2}
    a_{d+1}(x_{1},\ldots,x_{d},0)\ge a_{d}(x_{1},\ldots,x_{d}),\qquad
    \forall~x=(x_{1},\ldots,x_{d})\neq0,
  \end{equation}
  holds, the consistent set of random walks is called \emph{increasing}.
\end{definition}

Clearly, the concept of a consistent set of random walks is a special case of
the concept of an increasing consistent set of random walks. We have chosen
to retain the concept of a consistent set of random walks as a separate
category, since random walks defined by the Laplace operator possess
precisely this property.

Let $\phi_{d}(\theta)$ denote the Fourier transform of a random walk on
$\mathbb{Z}^{d}$ satisfying conditions $1^{(d)}$–$6^{(d)}$:
\[
  \phi_{d}(\theta):=\sum_{x\in\mathbb{Z}^{d}}a_{d}(x)e^{i\langle x, \theta
  \rangle}= \sum_{x\in\mathbb{Z}^{d}}a_{d}(x)(\cos{\langle x, \theta
  \rangle}-1),\qquad
  \theta\in[-\pi,\pi]^{d}.
\]
As is easily seen, $\phi_{d}(\theta)\leq 0$ for all
$\theta\in[-\pi,\pi]^{d}$. However, as shown in~\cite[Lemma~2.1.2]{Yar2025},
for a transient random walk (i.e., when axiom $4^{(d)}$ holds), a stronger
property applies:
\begin{equation}\label{eq:phi-positive}
  \phi_{d}(0)=0,\qquad \phi_{d}(\theta)<0,\quad
  \forall~\theta\in[-\pi,\pi]^{d},~\theta\neq0.
\end{equation}

\begin{lemma}\label{lem:1}
  If a collection of random walks with intensity matrices
  \[
    A_{d}=(a_{d}(x,y))_{x,y\in \mathbb{Z}^d},\quad d=1,2,\ldots\,,
  \]
  satisfying axioms $1^{(d)}$--$6^{(d)}$, forms a nested lattice-consistent
  collection of random walks, then
  \[
    |a_{d+1}(0)|\geq |a_{d}(0)|\qquad d=1,2,\ldots\,.
  \]
\end{lemma}

\begin{proof}
  By virtue of axioms $1^{(d+1)}$ and $2^{(d+1)}$,
  \[
    |a_{d+1}(0)|= \sum_{\substack{x\in\mathbb{Z}^{d+1}\\x\neq0}}a_{d+1}(x),
  \]
  whence, using the coordinate representation of the vector
  $x=(x_{1},\ldots,x_{d},x_{d+1})$ on the right-hand side of the resulting
  equality, we obtain
  \[
    |a_{d+1}(0)|=
    \sum_{\substack{x\in\mathbb{Z}^{d+1}\\x\neq0}}a_{d+1}(x_{1},\ldots,x_{d},x_{d+1}).
  \]
  Let us discard the terms in the last sum for which $x_{d+1}\neq0$. Then the
  corresponding sum does not increase, and for the remaining terms in it (in
  which $x_{d+1}=0$), the condition ${x=(x_{1},\ldots,x_{d},0)\neq0}$ implies
  the condition ${x=(x_{1},\ldots,x_{d})\neq0}$, and therefore
  \[
    |a_{d+1}(0)|\geq
    \sum_{\substack{x\in\mathbb{Z}^{d}\\x\neq0}}a_{d+1}(x_{1},\ldots,x_{d},0),
  \]
  whence, by virtue of condition~\eqref{eq:multidim2} and axioms $1^{(d)}$ and
  $2^{(d)}$, we obtain:
  \[
    |a_{d+1}(0)|\geq
    \sum_{\substack{x\in\mathbb{Z}^{d}\\x\neq0}}a_{d+1}(x_{1},\ldots,x_{d},0)\ge
    \sum_{\substack{x\in\mathbb{Z}^{d}\\x\neq0}}a_{d}(x_{1},\ldots,x_{d})=|a_{d}(0)|.
  \]

  Lemma~\ref{lem:1} is proved.
\end{proof}

\begin{lemma}\label{lem:2}
  The Fourier transforms $\phi_{d}(\theta)$ of an increasing
  sequence of nested
  lattice-consistent random walks decrease monotonically for each $\theta$ as
  $d$ increases:
  \[
    \phi_{d+1}(\theta_{1},\ldots,\theta_{d},\theta_{d+1})\le
    \phi_{d}(\theta_{1},\ldots,\theta_{d}),\qquad\forall~\theta_{i}\in[-\pi,\pi].
  \]
\end{lemma}

\begin{proof}
  By virtue of axiom $2^{(d+1)}$,
  \[
    a_{d+1}(0)= -\sum_{\substack{x\in\mathbb{Z}^{d+1}\\x\neq0}}a_{d+1}(x)
  \]
  and, therefore,
  \begin{align*}
    \phi_{d+1}(\theta)&
    =\sum_{x\in\mathbb{Z}^{d+1}}a_{d+1}(x)e^{i\langle x, \theta \rangle}=
    \sum_{\substack{x\in\mathbb{Z}^{d+1}\\x\neq0}}a_{d+1}(x)(\cos\langle
    x, \theta \rangle-1),
  \end{align*}
  whence, using the coordinate representations of the vectors
  \[
    x=(x_{1},\ldots,x_{d},x_{d+1}),\quad
    \theta=(\theta_{1},\ldots,\theta_{d},\theta_{d+1})
  \]
  on the right-hand side of the resulting equality, we obtain
  \begin{multline*}
    \phi_{d+1}(\theta_{1},\ldots,\theta_{d},\theta_{d+1})= \\
    =\sum_{\substack{x\in\mathbb{Z}^{d+1}\\x\neq0}}a_{d+1}(x_{1},\ldots,x_{d},x_{d+1})
    (\cos(x_{1}\theta_{1}+\cdots+x_{d}\theta_{d}+x_{d+1}\theta_{d+1}) -1).
  \end{multline*}
  Let us now discard the terms in the last sum for which $x_{d+1}\neq0$. Then
  the corresponding sum will not decrease, and for the remaining terms in it
  (in which ${x_{d+1}=0}$), the condition ${x=(x_{1},\ldots,x_{d},0)\neq0}$
  will imply the condition
  \[
    x=(x_{1},\ldots,x_{d})\neq0,
  \]
  and therefore
  \[
    \phi_{d+1}(\theta_{1},\ldots,\theta_{d},\theta_{d+1})\le
    \sum_{\substack{x\in\mathbb{Z}^{d}\\x\neq0}}a_{d+1}(x_{1},\ldots,x_{d},0)
    (\cos(x_{1}\theta_{1}+\cdots+x_{d}\theta_{d})-1).
  \]
  Hence, by virtue of~\eqref{eq:multidim2} we obtain
  \begin{align*}
    \phi_{d+1}(\theta_{1},\ldots,\theta_{d},\theta_{d+1})
    &
    \le\sum_{\substack{x\in\mathbb{Z}^{d}\\x\neq0}}a_{d+1}(x_{1},\ldots,x_{d},0)
    (\cos(x_{1}\theta_{1}+\cdots+x_{d}\theta_{d})-1)=
    \\
    & =\sum_{\substack{x\in\mathbb{Z}^{d}\\x\neq0}}a_{d}(x_{1},\ldots,x_{d})
    (\cos (x_{1}\theta_{1}+\cdots+x_{d}\theta_{d})-1)=
    \\
    & =\sum_{x\in\mathbb{Z}^{d}}a_{d}(x_{1},\ldots,x_{d})
    \cos(x_{1}\theta_{1}+\cdots+x_{d}\theta_{d})=
    \\
    & =\sum_{x\in\mathbb{Z}^{d}}a_{d}(x_{1},\ldots,x_{d})
    e^{i(x_{1}\theta_{1}+\cdots+x_{d}\theta_{d})}=
    \\
    & =\phi_{d} (\theta_{1},\ldots,\theta_{d}).
  \end{align*}
  Lemma~\ref{lem:2} is proven.
\end{proof}

As shown, for example, in~\cite{Yar2025}, for each $d=1, 2, \ldots$ and the
model described above, the probabilities of a particle transitioning over
time $t$ from point $x$ to point $y$ on $\mathbb{Z}^d$ can be represented as
$p_{d}(t,x,y)=p_{d}(t,0,y-x)=p_{d}(t,y-x)$, and as $h\downarrow 0$, they
satisfy the relations
\begin{align*}
  p_{d}(h, x)& = a_{d}(x)h + o(h), \quad x\not=0, \\
  p_{d}(h, 0)& = 1 + a_{d}(0)h + o(h),
\end{align*}
from which it follows that the probabilities $p_{d}(t,x)$ satisfy the
Kolmogorov backward equations
\begin{equation}\label{backeq}
  \frac{d p_{d}(t,x)}{dt} =
  \sum_{x'\in\mathbb{Z}^d}a_{d}(x')p(t,x'-x),\qquad p(0,x)=\delta(x),
\end{equation}
where $\delta(\cdot)$ is the Kronecker delta function on $\mathbb{Z}^d$.

In the sequel, we consider a random walk either with finite jump variance
(axiom~$5^d$)
\begin{equation}\label{P:1}
  \sum_{x \in \mathbb{Z}^d} |x|^2 a_{d}(x) < \infty,
\end{equation}
or with heavy tails, where the function $a_{d}(x)$ depends on
$\alpha\in(0,2)$---i.e., $a_{d}(x):=a_{d,\alpha}(x)$---and exhibits regular
behavior at infinity:
\begin{equation}\label{HT}
  a_{d,\alpha}(x) \sim \frac{H(x/|x|)}{|x|^{d+\alpha}},\qquad |x|\to~\infty,
\end{equation}
where $|\cdot|$ is the Euclidean norm on $\mathbb{R}^d$ and $H(\cdot)$ is a
positive, symmetric, continuous function on
$\mathbb{S}^{d-1}=\{x\in\mathbb{R}^d: |x|=1\}$. Under this assumption, the
jump variance becomes infinite~\cite{RytYar}

When condition~\eqref{HT} holds for $a_{d,\alpha}(x)$, we replace axiom $5^d$
with axiom $5^{d,\alpha}$:
\begin{equation}\label{Inf_Var}
  5^{d, \alpha}:\qquad\sigma^{2}=-
  \frac{1}{a_{d,\alpha}(0,0)}\sum_{x\in\mathbb{Z}^{d}}
  |x|^{2}a_{d,\alpha}(0,x)=\infty,
\end{equation}
i.e., the jumps have infinite variance, and will assume that all other axioms
hold for each fixed $\alpha$.

When using axiom~\eqref{Inf_Var}, we emphasize the dependence on $\alpha$ in
the notation for the transition intensities $a_{d,\alpha}(0,x)$ and write
them---as in Remark~\ref{R:1}---by omitting the zero argument:
\[
  a_{d,\alpha}(x):=a_{d,\alpha}(0,x),\qquad x\in\mathbb{Z}^{d}.
\]

\begin{remark}
  It is important to note that the proof of Lemmas~\ref{lem:1} and~\ref{lem:2}
  did not actually rely on axiom~$5^{d}$, and these lemmas remain valid under
  axiom $5^{d,\alpha}$ for any fixed ${\alpha\in (0,2)}$.
\end{remark}

\section{Green's function of a random walk}\label{S:0}
Applying the Fourier transform to equation~\eqref{backeq} (see,
e.g.,~\cite{Yar2025,Yar2024}), we obtain a representation for the Green's
function of the random walk:
\begin{equation}\label{Gxy}
  G_{\lambda}(x)=\frac{1}{(2\pi)^d}\int\limits_{[-\pi,\pi]^d}\frac{e^{i\langle
  x,\theta\rangle}}{\lambda-\phi_{d}(\theta)}d\theta,
\end{equation}
where $\langle\cdot,\cdot\rangle$ denotes the scalar product in
$\mathbb{R}^d$.

An alternative representation for the Green's function also holds (see,
e.g.,~\cite{Yar2025}); it expresses the function via the Laplace transform of
the transition probabilities $p_{d}(t,x)$ and takes the form:
\begin{equation}\label{Grin}
  G_{\lambda}(x):=
  \lim_{t\to \infty}\int_0^t e^{-\lambda u}p_{d}(u,x)\,du
  =\int_0^\infty e^{-\lambda u}p_{d}(u,x)\,du,\qquad \lambda \geq 0.
\end{equation}
This representation clarifies the probabilistic meaning of the function
$G_{0}(x)$: it represents the mean residence time of the random walk at point
$x\in \mathbb{Z}^{d}$ as $t\to \infty$. The asymptotic properties of the
Green's function $G_{\lambda}(x)$ for an arbitrary $\lambda$ were
investigated in~\cite{MolchanovYarovaya12}.

As is evident from~\eqref{Gxy} and~\eqref{Grin}, the Green's function
$G_{\lambda}(x)$ actually depends on the dimension $d$ as well; therefore, to
emphasize this dependence, we will henceforth use the notation
$G_{\lambda}(x) := G_{\lambda}(d, x)$:
\[
  G_\lambda(d,x)=
  \frac{1}{(2\pi)^d} \int\limits_{[-\pi,\pi]^d} \frac{e^{i\langle
  x,\theta\rangle}}
  {\lambda-\phi_{d}(\theta)}\,d\theta.
\]
According to~\eqref{eq:phi-positive}, $\phi_{d}(\theta) \leq 0$, and
consequently, at $x=0$, the Green's function
\begin{equation}\label{Glambda0}
  G_{\lambda}(d,0)=\frac{1}{(2\pi)^d} \int\limits_{[-\pi,\pi]^d} \frac{1}
  {\lambda-\phi_{d}(\theta)}\,d\theta
\end{equation}
decreases monotonically with respect to $\lambda$.

Recall the main properties of the Green function that play a key role in
analyzing the asymptotic properties of a random walk; for details,
see~\cite{Yar2024}.

Suppose condition~\eqref{P:1} holds. As $\lambda \to 0$, the function
$G_{\lambda}(d,0)$ tends to infinity in dimensions $d=1$ and $d=2$, whereas
for $d \geq 3$ it tends to a finite value $G_{0}(d,0)<\infty$.

Suppose condition~\eqref{HT} holds. As $\lambda \to 0$, the function
$G_{\lambda}(d,0)$ tends to infinity in dimension $d=1$ for any fixed
$\alpha\in [1,2)$, whereas it tends to a finite value $G_{0}(d,0)<\infty$ in
dimension $d=1$ for any fixed $\alpha\in (0,1)$ and in dimensions $d \geq 2$
for any fixed $\alpha\in (0,2)$.

Since $|G_{\lambda}(0)| < \frac{1}{\lambda}$, the function $G_{\lambda}(0)$
tends to zero as $\lambda \to \infty$ for any lattice dimension $d\in
\mathbb{N}$, under both condition~\eqref{P:1} and condition~\eqref{HT}.

Following, for instance,~\cite{RytYar}, we call a random walk
\emph{recurrent} if $G_0:=G_0(0,0)=\infty$, and \emph{transient} if
$G_0<\infty$. By virtue of~\eqref{Gxy}, establishing the recurrence of a
random walk requires upper and lower bounds on the Fourier transform
$\phi(\theta)$.

\section{Supercritical BRW with a single source}\label{S:5}
The question of the growth rate of particle numbers, $\lambda_0(d) > 0$,
arises not only in the analysis of simple random walks; the properties of
branching random walks also depend heavily on it. Therefore, in this section,
we recall the fundamental facts regarding the theory of branching random
walks that are necessary for our purposes.

We assume the presence of a single particle at the initial time $t=0$, which
begins a random walk on $\mathbb{Z}^d$ according to the rules described in
Section~\ref{S:RW}. At a specific point in $\mathbb{Z}^d$ (the branching
source), the initial particle may produce an arbitrary number of
offspring---including itself---or die. Without loss of generality, we assume
the branching source is located at the origin, $x=0$. We now turn to the
description of the branching mechanism. Branching at the source $x=0$ is
described by a Markov branching process defined by the infinitesimal
generating function $f(u)=\sum_{n=0}^{\infty}b_nu^n$ ($0\leqslant
u\leqslant1$), where $b_i$ represents the intensity of a particle dividing
into $i$ offspring (see, e.g.,~\cite{Yar2025} for details). We assume that
$b_n \ge 0$ for $n \neq 1$, $-\infty < b_1 < 0$, and
$\sum_{n=0}^{\infty}b_n=0$. Let us denote by
\begin{equation}
  \beta:=f'(1)=\sum_{n=0}^{\infty}nb_n=(-b_1)\left(-1+\sum_{n\not=1}^{\infty}
  n \frac{b_n}{-b_1}\right)
\end{equation}
the \emph{source intensity}, where the last sum characterizes the average
number of particles produced. Let $\beta^{(k)}= f^{(k)}(1)<\infty$ for $k=1$
or $k=2$.

To define the BRW process, we combine a random walk and a branching mechanism
in the manner described, for example, in~\cite{Yar2025}. Then (see,
e.g.,~\cite{Yar2025}), the random time until branching at the origin follows
an exponential distribution with parameter $-a_{d}(0)-b_1$. Subsequently, the
particle transforms into a random number of particles or moves to another
lattice site. The new particles evolve according to the same rules,
independently of one another and of the entire history of the process.

Of primary interest in the study of the described model is the behavior of
the local particle numbers $\mu_t(y)$ at an arbitrary point $y \in
\mathbb{Z}^d$, the total population size $\mu_t = \sum_{y \in \mathbb{Z}^d}
\mu_t(y)$, and their moments $m_n(t, x, y) = \mathsf{E}_x \mu_t^n(y)$ and
$m_n(t, x) = \mathsf{E}_x \mu_t^n$, where $\mathsf{E}_x$ denotes the
expectation given that, at the initial time, the system contained a single
particle located at point $x$. Using the derivation scheme
from~\cite{Yar2025}, we can proceed to linear differential equations for
$m_1(t,x)$ and $m_1(t,x,y)$ in the space $\ell^p(\mathbb{Z}^d)$, $1\leqslant
p\leqslant\infty$:
\begin{align*}
  \frac{d m_1(t,x)}{dt}   & = (\mathcal{H}_{\beta}\,m_1(t,
  \cdot))(x),  \quad m_1(0,x) = 1,                 \\
  \frac{d\,m_1(t,x,y)}{dt}& = (\mathcal{H}_{\beta}\,m_1(t, \cdot,
  y))(x),  \quad m_1(0, x, y) = \delta_x(y),
\end{align*}
where the operator $\mathcal{H}_{\beta}$ is defined by the equality
\[
  \mathcal{H}_{\beta}:=\mathcal{H}_{\beta,d}=\mathcal{A}_{d}+\beta\mathcal{V}_{0},
\]
in which $(\mathcal{V}_{0}u)(x)=u(0)\delta_{0}(x)$, and the operator
$\mathcal{A}_{d}: \ell^p(\mathbb{Z}^d)\rightarrow\ell^p(\mathbb{Z}^d)$,
$1\leqslant p\leqslant\infty$, is defined by the equality
\[
  (\mathcal{A}_{d}f)(x)=\sum_{x'}a_{d}(x,x')f(x').
\]
In the space $\ell^2(\mathbb{Z}^d)$, the operator $\mathcal{A}_{d}$ is
self-adjoint.

The behavior of the branching random walk depends significantly on the
recurrence property. Let $\beta_c$ denote the minimum source intensity such
that, for $\beta > \beta_c$, the spectrum of the operator
$\mathcal{H}_{\beta}$ contains an eigenvalue $\lambda_{0}(d) > 0$ (see,
e.g.,~\cite{Yar2025}); this eigenvalue is a solution to the equation
\begin{equation}\label{E:Gd}
  G_{\lambda}(d,0)=\beta^{-1}
\end{equation}
where $\beta > 0$. We call a BRW subcritical if $\beta < \beta_c$, critical
if $\beta = \beta_c$, and supercritical if $\beta > \beta_c$. For a BRW with
a single branching source and $\beta > \beta_c$, limit theorems have been
proven (see, e.g.,~\cite{Yar2024,Yar2025}) establishing the exponential
growth of particle numbers with parameter $\lambda_{0}(d) > 0$, both at
individual lattice sites and across the entire lattice. It is shown
in~\cite{Yar2024} that for the single-source case, the following relations
hold: if $G_0 = \infty$, then $\beta_c = 0$, and if $G_0 < \infty$, then
$\beta_c = G_0^{-1} = 1/G_0$.

For the BRW models under consideration, the transition from finite to
infinite jump variance gives rise to new effects. Thus, although the
fundamental differential and integral equations for the generating functions
and the equations for the particle count moments remain unchanged when
condition~\eqref{HT} is satisfied, the properties and asymptotic behavior of
the functions $\phi(\theta)$ and $p(t,x,y)$---and, consequently, of the
functions $G_{\lambda}(x,y)$ and $m_1(t,x,y)$---undergo qualitative changes.
For instance, in the case of infinite jump variance, the condition
$G_0<\infty$ can hold in dimension $d=1$ for $\alpha\in(0,1)$ and in
dimensions $d\ge2$ for $\alpha\in(0,2)$, whereas in the case of finite jump
variance, $G_0<\infty$ holds only in dimensions $d\ge3$. This implies that
for the lattices $\mathbb{Z}^{1}$ and $\mathbb{Z}^2$, the value of $\beta_c$
can be strictly positive; that is, for $d=1$ or $d=2$, a heavy-tailed
branching random walk (unlike the finite-variance case) can be subcritical
even when the branching process at the source is supercritical.

\section{Main Results}\label{S:4}

In this section, we prove the main assertions of the paper:
Theorem~\ref{T:6}, concerning the monotonic decrease with respect to the
dimension $d$ of the eigenvalues of the evolution operator for a
supercritical branching random walk, and Theorem~\ref{T:7}, regarding the
asymptotic expansion of the principal eigenvalue for a supercritical
branching random walk in the regime of high particle-multiplication source
intensity.

We begin by analyzing the growth of the particle numbers $\lambda_0(d) > 0$
as a function of $d$, which is equivalent to analyzing the growth of the
solution to the equation $G_{\lambda}(d,0) = \beta^{-1}$.

\begin{lemma}\label{L:Gd} For each $\lambda > 0$, the Green's
  function $G_{\lambda}(d,0)$ of an increasing nested
  lattice-consistent set of random walks is monotonically decreasing in $d$.
\end{lemma}
\begin{proof}
  For an arbitrary $d$, we estimate the denominator of the integrand
  in~\eqref{Glambda0} using Lemma~\ref{lem:2} as follows:
  \[
    \lambda-\phi_{d+1}(\theta_{1},\ldots,\theta_{d},\theta_{d+1})\ge
    \lambda-\phi_{d}(\theta_{1},\ldots,\theta_{d}),\qquad\forall~\theta_{i}\in[-\pi,\pi].
  \]
  From this, we obtain
  \begin{align*}
    G_{\lambda}(d+1,0)& =\frac{1}{(2\pi)^{d+1}}
    \int\limits_{[-\pi,\pi]^{d+1}} \frac{1}
    {\lambda-\phi_{d+1}(\theta_{1},\ldots,\theta_{d},\theta_{d+1})}\,d\theta_{1}\cdots
    d\theta_{d}\,d\theta_{d+1}\le
    \\
    & \le\frac{1}{(2\pi)^{d+1}} \int\limits_{[-\pi,\pi]^{d+1}} \frac{1}
    {\lambda-\phi_{d}(\theta_{1},\ldots,\theta_{d})}\,d\theta_{1}\cdots
    d\theta_{d}\,d\theta_{d+1}.
  \end{align*}
  Let us express the right-hand side of the resulting inequality as
  an iterated
  integral:
  \[
    G_{\lambda}(d+1,0)\le\frac{1}{(2\pi)^{d+1}}\int_{-\pi}^{\pi}
    \left(\int\limits_{[-\pi,\pi]^{d}} \frac{1}
      {\lambda-\phi_{d}(\theta_{1},\ldots,\theta_{d})}d\theta_{1}\ldots
    d\theta_{d} \right)d\theta_{d+1}.
  \]
  Here, the expression in parentheses is simply
  $(2\pi)^{d}\,G_{\lambda}(d,0)$;
  therefore,
  \[
  G_{\lambda}(d+1,0)<\frac{1}{(2\pi)^{d+1}}\int_{-\pi}^{\pi}
  (2\pi)^{d}G_{\lambda}(d,0)\,d\theta_{d+1}.
  \]
  Since the integrand does not depend on $\theta_{d+1}$, we obtain
  \[
  G_{\lambda}(d+1,0)<\frac{G_{\lambda}(d,0)}{2\pi}\int_{-\pi}^{\pi}\,d\theta_{d+1}=G_{\lambda}(d,0).
  \]

  Lemma~\ref{L:Gd} is proved.
\end{proof}

\begin{figure}[htbp!]
  \centering\includegraphics[width=0.8\textwidth,clip]{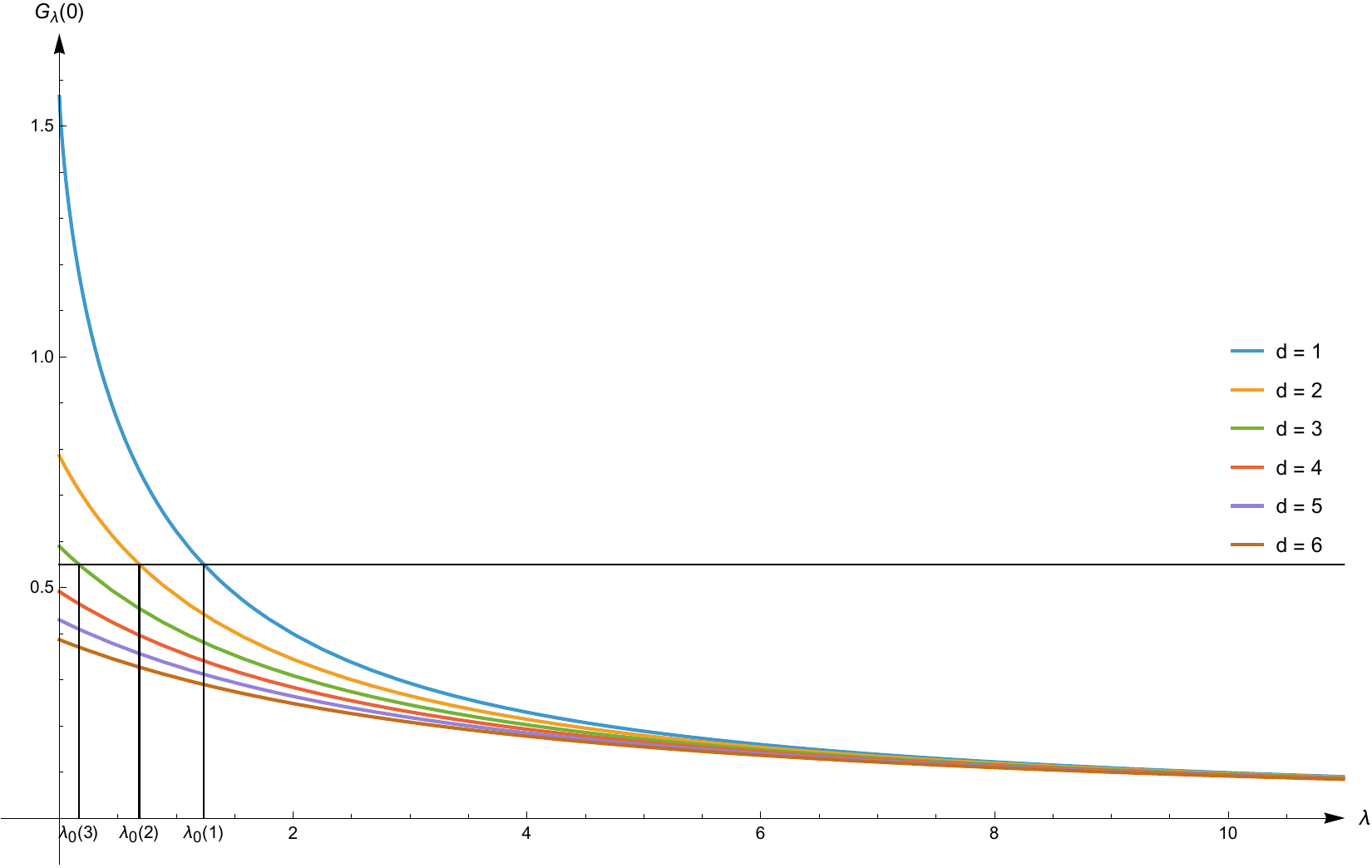}
  \caption{Behavior of the solution $\lambda_{0}(d)$ to
    equation~\eqref{E:Gd} as a function of
    the dimension of $\mathbb{Z}^{d}$ under axiom $5^{d, \alpha}$ for
  each $\alpha\in (0,1)$.}\label{F:3}
\end{figure}
\begin{figure}[htbp!]
  \centering\includegraphics[width=0.8\textwidth,clip]{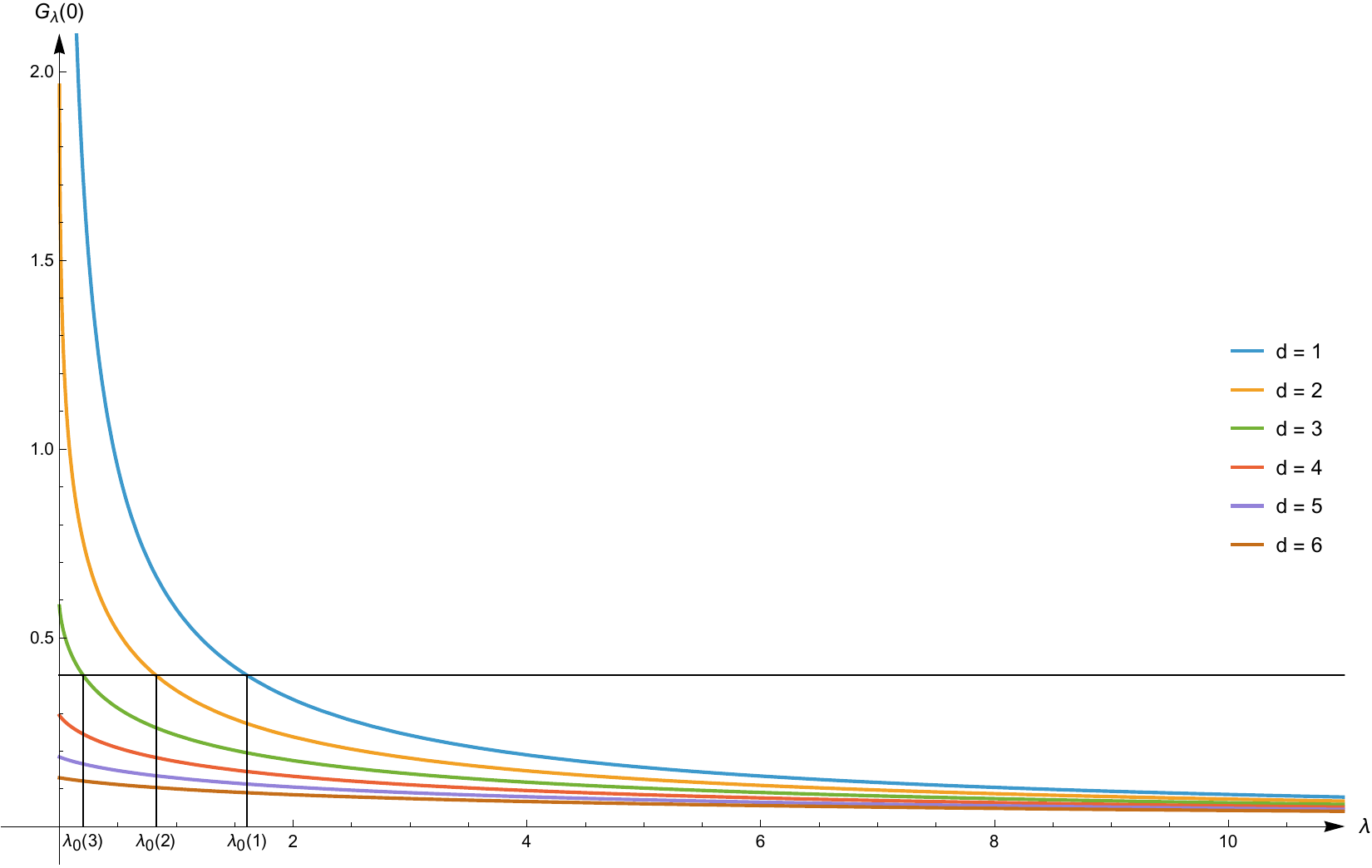}
  \caption{Behavior of the solution $\lambda_{0}(d)$ to
    equation~\eqref{E:Gd} as a function of
    the dimension of $\mathbb{Z}^{d}$ under axiom $5^{d, \alpha}$ for
  each $\alpha\in [1,2)$.}\label{F:4}
\end{figure}

We are now ready to prove the first of the main assertions of this section.
\begin{theorem}\label{T:6}
  Suppose that one of the conditions~\eqref{P:1} or~\eqref{HT} holds. Then
  there exists an integer $d(\beta)$ such that for each
  $d=1,2,\ldots,d(\beta)$, equation~\eqref{E:Gd} has a unique
  positive solution
  $\lambda_{0}:=\lambda_{0}(d, \beta)\in (0,\beta)$. Moreover, the function
  $\lambda_{0}(d, \beta)$ is strictly monotonically decreasing in $d$, i.e.,
  \[
    0<\lambda_{0}(n,\beta)<\ldots
    \lambda_{0}(d+1,\beta)<\lambda_{0}(d,\beta)
    <\ldots<\lambda_{0}(1,\beta) <\beta.
  \]
\end{theorem}
\begin{proof}
  By virtue of~\eqref{Glambda0}, for each $d=1,2,\ldots$, the
  function $G_{\lambda}(d,0)$ is monotonically decreasing and tends
  to zero as $\lambda\to\infty$. Therefore, for a given $\beta>0$,
  there exists an integer $d(\beta)$ such that for
  $d>d(\beta)$, equation~\eqref{E:Gd} no longer has a positive solution
  $\lambda_{0}(d)$, whereas it does have a solution at $d=d(\beta)$. Then, by
  Lemma~\ref{L:Gd}, equation~\eqref{E:Gd} has a solution
  $\lambda_{0}(d)>0$ for
  each $d=1,2,\ldots,d(\beta)$; that is, for these values of $d$, the
  inequality $\beta>G_{0}^{-1}(0,d)$ holds~\cite{Yar2025}.
  Moreover, due to the
  monotonicity of $G_{\lambda}(d,0)$ with respect to $\lambda$, the solution
  $\lambda_{0}(d)$ is unique.

  By virtue of~\eqref{eq:phi-positive}, the function $\phi_{d}(\theta)$ is
  strictly negative for almost all $\theta\in[\pi,\pi]^{d}$. Then it follows
  from~\eqref{Glambda0} that $G_{\lambda}(d,0)<\lambda^{-1}$, and therefore,
  \[
    \lambda_{0}(1,\beta)< G_{\lambda_{0}}(d,0)^{-1}=\beta.
  \]

  To prove the remaining inequalities, it suffices to show that the inequality
  \begin{equation}\label{E:inlambda}
    \lambda_{0}(d+1,\beta)<\lambda_{0}(d,\beta)
  \end{equation}
  holds for each $d=1,2,\ldots,n-1$.

  Finally, we note that, by Lemma~\ref{L:Gd}, for any value $\lambda>0$, the
  function $G_{\lambda}(d,0)$ is monotonically decreasing in $d$. This,
  combined with the obvious fact that $G_{\lambda}(d,0)$ monotonically
  decreases with respect to $\lambda$, implies relation~\eqref{E:inlambda},
  thereby completing the proof of Theorem~\ref{T:6}.
\end{proof}

\begin{corollary}\label{C:1}
  Suppose condition~\eqref{P:1} holds; then, for dimensions $d=1$ and $d=2$,
  equation~\eqref{E:Gd} has a unique positive solution for any
  value $\beta>0$.
  Moreover,
  \[
    \lim_{\beta \to 0}(\lambda_{0}(1,\beta)-\lambda_{0}(2,\beta))=0.
  \]
\end{corollary}
\begin{proof} The existence of solutions for all $\beta > 0$
  follows from the fact that, in
  dimensions $d = 1$ and $d = 2$, the function $G_{\lambda}(d,0)$ increases
  without bound as $\lambda\to0$; see, for example,~\cite{Yar2025}.
  Furthermore, by Theorem~\ref{T:6}, the following inequalities hold for all
  $\beta > 0$:
  \[
    0 < \lambda_0(2, \beta) < \lambda_0(1, \beta) < \beta.
  \]
  Letting $\beta$ tend to zero, we obtain the assertion of
  Corollary~\ref{C:1}.
\end{proof}

\begin{corollary}\label{C:2}
  Suppose condition~\eqref{HT} holds. Then, for dimension $d=1$ and any fixed
  $\alpha\in [0,1)$, equation~\eqref{E:Gd} has a unique positive solution for
  any $\beta>0$; in the other cases, a solution exists if $G_{\lambda}(d,0)>
  \beta^{-1}$.
\end{corollary}
\begin{proof} The existence of solutions for all $\beta > 0$
  follows from the fact that, in dimensions $d = 1$ with $\alpha
  \in [1,2)$, the Green function at zero $G_\lambda(d,0)$ diverges
  as $\lambda \to 0$ (see, e.g.,~\cite{Yar2024}, and
  Fig.~\ref{F:4}). In the other
  cases, $G_\lambda(d,0)<\infty$ (see Figs.~\ref{F:3}, and~\ref{F:4}).
\end{proof}

\begin{theorem}\label{T:7}
  If condition~\eqref{P:1} holds, then for each $d\in\mathbb{N}$,
  the following
  representation is valid:
  \[
    \lambda_{0}(d,\beta)=\beta+a_d(0) + O(\beta^{-1}) \quad \text{as}
    \quad \beta\to \infty.
  \]

  If condition~\eqref{HT} holds, then for each $d\in\mathbb{N}$ and each fixed
  $\alpha \in (0,2)$, the following representation is valid:
  \[
    \lambda_{0}(d,\beta)=\beta+a_{d,\alpha}(0) + O(\beta^{-1}) \quad \text{as}
    \quad \beta\to \infty.
  \]
\end{theorem}

\begin{proof}
  The constructions presented below hold true whether condition~\eqref{P:1} or
  condition~\eqref{HT} is satisfied. Accordingly, when condition~\eqref{HT}
  holds, we will omit the index $\alpha$. In both cases, by virtue
  of~\eqref{Glambda0}, equation~\eqref{E:Gd} takes the form
  \begin{equation}\label{eq:lambda-beta}
    \frac{1}{(2\pi)^d}\int\limits_{[-\pi,\pi]^d}
    \frac{1}{\lambda-\phi_{d}(\theta)}\,d\theta=\frac{1}{\beta}.
  \end{equation}
  Our goal is to find a solution $\lambda$ to this equation as a function of
  the variable $\beta$; that is, to express $\lambda$ in the form $\lambda =
  \lambda(d, \beta)$. Unfortunately, determining the asymptotic behavior of
  $\lambda(d, \beta)$ as $\beta \to \infty$ is difficult, because
  in this case,
  the functions involved cannot be expanded into Taylor series. Therefore, we
  perform a change of variables in equation \eqref{eq:lambda-beta} by setting
  \begin{equation}\label{eq:varchange}
    \beta = \tau^{-1}, \quad \lambda = \tau^{-1}\gamma.
  \end{equation}
  In terms of the new variables $\tau$ and $\gamma$, equation
  \eqref{eq:lambda-beta} takes the form
  \begin{equation}\label{E:DefF}
    \tilde{F}(\gamma, \tau) := \frac{1}{(2\pi)^d}
    \int\limits_{[-\pi, \pi]^d} \frac{1}{\gamma -
    \tau\phi_d(\theta)} \, d\theta = 1.
  \end{equation}
  Let $\gamma = \gamma(d, \tau)$ be the solution to this equation,
  expressed as
  a function of $\tau$. Then, by virtue of \eqref{eq:varchange}, the functions
  $\lambda(d, \beta)$ and $\gamma(d, \tau)$ are related by the
  following simple
  expressions:
  \[
    \gamma(d, \tau) = \tau\lambda(d, \tau^{-1}) \quad
    \Longleftrightarrow \quad \lambda(d, \beta) = \beta\gamma(d, \beta^{-1}).
  \]
  Consequently, the problem of finding the asymptotic behavior of $\lambda(d,
  \beta)$ as $\beta \to \infty$ reduces to the analogous problem of
  finding the
  asymptotic behavior of $\gamma(d, \tau)$ as $\tau \to 0$. And the final
  problem, as shown below, reduces to the standard problem of finding an
  implicit function.

  Indeed, since
  \[
    \tilde{F}(1, 0)=1,
  \]
  equation~\eqref{E:DefF} obtained above has a unique solution $\gamma=1$ at
  $\tau=0$. Moreover, the function $\tilde{F}(\gamma, \tau)$ satisfies the
  conditions of the classical implicit function theorem~\cite{KF81:r}:
  $\tilde{F}(\gamma, \tau)$ is infinitely differentiable with respect to both
  variables, and
  \[
    \tilde{F}^{\prime}_{\gamma}(1, 0)=-1\neq 0,
  \]
  which is crucial for applying the implicit function theorem. In
  this case, by
  the implicit function theorem, equation~\eqref{E:DefF} defines
  $\gamma$ as an
  infinitely continuously differentiable function of $\tau$ in some
  neighborhood of the point $\tau=0$:
  \begin{equation}\label{E:Fimplicit}
    \tilde{F}(\gamma (\tau), \tau)\equiv 1 \quad \text{with} \quad
    \gamma (0)=1.
  \end{equation}
  To find $\gamma^{\prime}_{\tau}(0)$ we differentiate
  identity~\eqref{E:Fimplicit}:
  \[
    \tilde{F}^{\prime}_{\gamma}(\gamma (\tau),
    \tau)\gamma^{\prime}(\tau)+\tilde{F}^{\prime}_{\tau}(\gamma
    (\tau), \tau)\equiv 0,
  \]
  whence for $\tau=0$ we get
  \[
    \tilde{F}^{\prime}_{\gamma}(1,
    0)\gamma^{\prime}(0)+\tilde{F}^{\prime}_{\tau}(1, 0)=0.
  \]
  Therefore,
  \[
    \gamma^{\prime}(0)=-\frac{\tilde{F}^{\prime}_{\tau}(1,
    0)}{\tilde{F}^{\prime}_{\gamma}(1, 0)},
  \]
  which gives us the following representation of the function $\gamma(\tau)$:
  \begin{equation}\label{E:gamma tau}
    \gamma(\tau)=\gamma(0)+\gamma^{\prime}(0)\tau+O(\tau^{2})=1+\gamma^{\prime}(0)\tau+O(\tau^{2}).
  \end{equation}
  To find $\gamma^{\prime}(0)$, it remains for us to calculate two quantities:
  $\tilde{F}^{\prime}_{\gamma}(1,0)$ and $\tilde{F}^{\prime}_{\tau}(1,0)$.
  Based on definition~\eqref{E:DefF} of the function
  $\tilde{F}(\gamma, \tau)$,
  we have
  \[
    \tilde{F}^{\prime}_{\gamma}(\gamma, \tau)=-\frac{1}{(2\pi)^d}
    \int\limits_{[-\pi,\pi]^d} \frac{1}
    {(\gamma-\tau\phi_{d}(\theta))^{2}}\,d\theta,
  \]
  and, consequently, $\tilde{F}^{\prime}_{\gamma}(1, 0)=-1$. Again, based on
  definition~\eqref{E:DefF}, we have
  \[
    \tilde{F}^{\prime}_{\tau}(\gamma, \tau)=\frac{1}{(2\pi)^d}
    \int\limits_{[-\pi,\pi]^d}
    \frac{\phi_{d}(\theta)}
    {(\gamma-\tau\phi_{d}(\theta))^{2}}\,d\theta,
  \]
  where
  \begin{align}\nonumber
    \tilde{F}^{\prime}_{\tau}(1, 0)& =
    \frac{1}{(2\pi)^d} \int\limits_{[-\pi,\pi]^d}
    \phi_{d}(\theta)\,d\theta=
    \\
    \nonumber & =\frac{1}{(2\pi)^d}
    \int\limits_{[-\pi,\pi]^d} \sum_{x\neq
    0}a_d(x)(\cos(x_{1}\theta_{1}+\ldots+x_{d}\theta_{d})-1)\,d\theta= \\
    \nonumber & = \frac{1}{(2\pi)^d}\sum_{x\neq 0}
    a_d(x)
    \int\limits_{[-\pi,\pi]^d}\cos(x_{1}\theta_{1}+\ldots+x_{d}\theta_{d})
    \,d\theta_{1}\cdots d\theta_{d} -
    \\
    \label{F:prime tau}
    & \quad-\sum_{x\neq 0}a_d(x).
  \end{align}
  Here, each integral under the summation sign vanishes, and the last sum
  equals $-a_{d}(0)$ by virtue of axiom $2^{(d)}$;
  consequently,~\eqref{F:prime
  tau} takes the form:
  \[
    \tilde{F}^{\prime}_{\tau}(1, 0)=a_d(0).
  \]
  Thus, from~\eqref{E:gamma tau}, we obtain
  \[
    \gamma(\tau)=1+a_d(0)\tau+O(\tau^{2}),
  \]
  where, we recall, $a_d(0)<0$. From this, reverting from the variables
  $\gamma$ and $\tau$ to the variables $\lambda$ and $\beta$, we obtain
  \[
    \lambda(d,\beta)=\beta\gamma(\beta^{-1})=\beta+a_d(0)+O(\beta^{-1}).
  \]
  Theorem~\ref{T:7} is proved.
\end{proof}

\begin{corollary}\label{C:3}  Suppose condition~\eqref{P:1} holds;
  then, as $\beta \to \infty$, for each
  $d\in \mathbb{N}$, the following relation holds:
  \[
    \lim_{\beta \to
    \infty}(\lambda_{0}(1,\beta)-\lambda_{0}(d,\beta))=a_{1}(0)-a_d(0).
  \]

  If, however, condition~\eqref{HT} holds, then as $\beta \to
  \infty$, for each
  $d\in \mathbb{N}$ and each $\alpha\in (0,2)$, the following relation holds:
  \[
    \lim_{\beta \to
    \infty}(\lambda_{0}(1,\beta)-\lambda_{0}(d,\beta))=a_{1,\alpha}(0)-a_{d,\alpha}(0).
  \]
\end{corollary}

\begin{proof}
  The assertion of Corollary~\ref{C:3} follows directly from
  Theorem~\ref{T:7}.
\end{proof}

\section{Conclusion}\label{Concl}

This paper proposes a new approach to describing families of random walks on 
lattices of varying dimensions, based on the concept of a set of random walks 
that is consistent with respect to nested lattices. Within this framework, it 
is proven that when transitioning from $\mathbb{Z}^d$ to $\mathbb{Z}^{d+1}$, 
the intensity $|a_d(0)|$ does not decrease, whereas the Fourier transform 
$\phi_d(\theta)$ decreases monotonically with $d$. These properties hold true 
regardless of whether the jump variance is finite or infinite under 
condition~\eqref{HT}. 

The main result of the paper is an asymptotic formula for the positive
eigenvalue of the mean-population operator in a supercritical branching
random walk with a single branching source:
\[
  \lambda_0(d,\beta) = \beta + a_d(0) + O(\beta^{-1}), \quad \beta\to\infty,
\]
which holds for a wide class of transition intensities, including models
where the random walk jump variance may be infinite.


\end{document}